\documentclass[]{amsart}
\usepackage{amsthm,wrapfig,amsmath,amssymb,amsfonts,setspace,verbatim,graphicx,mathtools,mathrsfs,commath,float,mdframed,frame,xcolor}
\usepackage{enumitem}
\usepackage[hidelinks]{hyperref}
\usepackage{ulem}%
\usepackage[margin=1.45in]{geometry}
\newtheorem{ut}{Theorem}
\usepackage[T1]{fontenc}

\newtheorem{up}{Proposition}
\newtheorem{ul}{Lemma}

\usepackage{calligra}

\DeclareMathAlphabet{\mathcalligra}{T1}{calligra}{m}{n}
\DeclareFontShape{T1}{calligra}{m}{n}{<->s*[2.2]callig15}{}

\newtheorem{uq}{Question}

\theoremstyle{remark}

\begin{document}

\title[Approximations by disjoint continua]{Approximations by disjoint continua and \\ a positive entropy conjecture}

\author[D.S. Lipham]{David S. Lipham}
\address{Department of Mathematics, Auburn University at Montgomery, Montgomery 
AL 36117, United States of America}
\email{dsl0003@auburn.edu,dlipham@aum.edu}

\begin{abstract}E.D. Tymchatyn constructed a hereditarily locally connected continuum which can be approximated by a sequence of mutually disjoint arcs. We show the example re-opens a conjecture of G.T. Seidler and H. Kato about continua which admit positive entropy homeomorphisms. We prove  that every indecomposable semicontinuum can be approximated by a sequence of disjoint subcontinua, and no composant of an indecomposable continuum can be embedded into a Suslinian continuum. We also prove that if $Y$ is a hereditarily unicoherent Suslinian continuum, then  there exists $\varepsilon>0$ such that every two  $\varepsilon$-dense subcontinua of $Y$ intersect. \end{abstract}

\maketitle

\section{Introduction}

In 1990, G.T. Seidler proved that every homeomorphism on a regular curve has zero topological entropy \cite[Theorem 2.3]{sei}.  He conjectured: \textit{Every homeomorphism on a rational curve has zero topological entropy} \cite[Conjecture 3.4]{sei}.  In 1993, H. Kato asked a related question: \textit{If $f:X\to X$ is a homeomorphism of a continuum $X$, and the topological entropy of $f$ is positive, is $X$ non-Suslinian?} \cite[Question 1]{kat}. A positive answer to the latter implies the former because every rational continuum is Suslinian (see Section 2 for definitions).  

In 2016, a positive answer to Kato's question was announced  \cite[Corollary 27]{mou2}. The proof in \cite{mou2} relies on \cite[Theorem 17]{mou2}, which is stronger than:

\begin{up}[{\cite[Theorem 30]{mou1}}]  If $(Y,d)$ is a continuum and $\{X_n\}_{n=0}^\infty$ is a collection of mutually disjoint subcontinua of $Y$ such that $$d_H(X_n,Y):=\sup_{y\in Y}d(y,X_n)\to 0\text{ as }n\to\infty,$$ then $Y$ is non-Suslinian.\end{up} \noindent Unfortunately, Proposition 1 is false by \cite[Example 3]{tym}. The example, constructed by E.D. Tymchatyn in 1983,  is a hereditarily locally connected continuum which is the closure of a first category ray and is therefore the limit of a sequence of  disjoint arcs in Hausdorff distance $d_H$. It is well-known that hereditarily locally connected continua are rational and Suslinian.  Seidler's conjecture thus remains an open problem. We remark that  \cite[Theorem 2.8]{newkat} is also contingent on  \cite[Corollary 27]{mou2}.

The hypothesis of Proposition 1 defines what it means for a continuum to be approximated by a sequence of disjoint subcontinua.  Approximations of continua from within were originally studied by J. Krazinkiewicz and P. Minc  in \cite{minc}.   They proved:  \textit{If $Y$ is a hereditarily unicoherent plane continuum which contains disjoint $\varepsilon$-dense subcontinua for each $\varepsilon>0$, then $Y$ contains an indecomposable continuum \cite[Theorem 1]{minc}.}  S. Curry later proved that if the continuum $Y$ is tree-like then it can be written as the union of two indecomposable subcontinua \cite[Theorem 5]{cur}. These results do not extend to non-planar continua, as there exists a (hereditarily decomposable) dendroid which is  approximated by a sequence of disjoint subcontinua \cite[Example 2]{minc}.  We can, however, reach the weaker conclusion that $Y$ is non-Suslinian. 

\begin{ut}If $Y$ is a hereditarily unicoherent continuum which contains disjoint $\varepsilon$-dense subcontinua for every $\varepsilon>0$, then $Y$ is non-Suslinian.\end{ut}

So Proposition 1 is true for all hereditarily unicoherent continua, including all tree-like continua. % In light of the  arguments in \cite{mou2}, this suggests that Seidler's conjecture is at least true for  all tree-like continua. Proved by Mouron separately 2007?

 Next, we will investigate the role of 
continuum-wise connected spaces, or \textit{semicontinua}, in approximations.   
We  show that each   continuum $Y$  which densely contains an indecomposable semicontinuum $X$ can be approximated by a sequence of disjoint continua (in the sense of Proposition 1). This will be a consequence of Theorem 2.  Further, if $X$ is homeomorphic to a composant of an indecomposable continuum, then $Y$ is non-Suslinian. An even  stronger result is stated in Theorem 3.

\begin{ut}If  $(X,d)$ is an indecomposable semicontinuum, then there is a sequence of pairwise disjoint continua $K_0,K_1,K_2,\ldots\subsetneq X$ such that $K_n\to X$ in the Vietoris topology. In particular, $d(x,K_n)\to 0$ for each $x\in X$.\end{ut}

 \begin{ut}If $X$ is homeomorphic to a composant of an indecomposable continuum, then $X$ cannot be embedded into a  Suslinian continuum. Moreover, every compactification of $X$ contains  $\mathfrak c=|\mathbb R|$ pairwise disjoint dense  semicontinua.
\end{ut}

Theorem 2 extends \cite[Corollary 1.2]{lip4}. In Theorem 3, ``Suslinian''  cannot be replaced with ``hereditarily decomposable'', as there exists a hereditarily decomposable plane continuum which homeomorphically contains  composants of the bucket-handle continuum. See \cite[Section 5]{mou} and \cite[Section 1.1]{lip4}.   
We also remark that Theorem 3 is false for indecomposable connected sets in general; there exists an indecomposable connected subset of the plane which can be embedded into a Suslinian continuum    \cite[Examples 2 and 4]{lip5}.

\section{Preliminaries}

All spaces under consideration are separable and metrizable. 

A \textit{continuum} is a compact connected metrizable space with more than one point.  
      An \textit{arc} is a continuum homeomorphic to $[0,1]$.  A connected set $X$ is \textit{decomposable} if $X$ can be written as the union of two proper closed connected subsets.   Otherwise, $X$ is \textit{indecomposable}. 

An \textit{indecomposable semicontinuum} is a continuum-wise connected space which cannot be written as the union of two proper closed connected subsets. 
      A \textit{composant} of a continuum  is the union of all proper subcontinua that contain a given point.  Observe that each composant of an indecomposable continuum is an indecomposable semicontinuum. The class of spaces which are homeomorphic to composants of indecomposable continua includes all \textit{singular dense meager composants}; see  \cite{lip4}. 
       %A \textit{dendrite} is a locally connected continuum which contains no simple closed curve.  

A continuum $Y$ is:

\begin{itemize}
\item \textit{hereditarily unicoherent} if $H\cap K$ is connected for every two subcontinua $H$ and $K$;
\item \textit{regular} if $Y$ has a basis of open sets with finite boundaries;
\item \textit{hereditarily locally connected} if every subcontinuum of $Y$ is locally connected;
\item  \textit{rational} if $Y$ has a basis of open sets with countable boundaries;  
\item \textit{Suslinian} if $Y$ contains no uncountable collection of pairwise disjoint subcontinua  \cite{lel}; and 
\item \textit{hereditarily decomposable} if every subcontinuum of $Y$ is decomposable.
\end{itemize}
For continua, it is well-known that: 
$$ \text{regular }\Rightarrow \text{ hereditarily locally connected } \Rightarrow \text{ rational } \Rightarrow \text{ Suslinian}$$ $$\Rightarrow \text{hereditarily decomposable}\Rightarrow \text{one-dimensional}.$$    One-dimensional continua are  frequently called \textit{curves}.

For any topological space $X$ we let $2^X$ denote the set of non-empty closed subsets of $X$. 
A sequence $(A_n)\in [2^X]^\omega$ converges to $X$ in the Vietoris topology provided for every finite collection of non-empty open sets $U_1,\ldots,U_k\subset X$   there exists $N<\omega$ such that $$A_n\in \textstyle\langle U_1,\ldots,U_k\rangle:=\{A\in 2^X:A\cap U_i\neq\varnothing \text{ for each } i\leq k\}$$ 
 for all $n\geq N$.  If $(X,d)$ is compact,  then  $A_n\to X$  in the Vietoris topology  if and only if $\lim\limits_{n\to\infty}d_H(A_n,X)=0$.   Here $d_H(A_n,X)=\sup_{x\in X}d(x,A_n)$ is the Hausdorff distance between $A_n$ and $X$. A  subset $E$ of $X$ is  said to be \textit{$\varepsilon$-dense} if $d_H(E,X)<\varepsilon$, i.e.\  if $E$ intersects every ball of radius $\varepsilon$ in $X$.

\section{Proofs}

\subsection*{Proof of Theorem 1}Let $Y$ be a hereditarily unicoherent Suslinian continuum.  We will find $\varepsilon>0$ such that every two $\varepsilon$-dense subcontinua of $Y$ intersect. 

Note that $Y$ is decomposable, so there exist proper subcontinua $H$ and $K$ of $Y$ such that $Y=H\cup K$.   Let $W$ be an open subset of $Y$ such that $H\cap K\subset W$ and $\overline W\neq Y$. Each connected component  of  $Y\setminus \overline W$ contains a non-degenerate continuum by \cite[Lemma 6.1.25]{eng}, so the Suslinian property of $Y$ implies that the set of connected components of $Y\setminus \overline W$ is countable. By Baire's theorem there is a component $C$ of $Y\setminus \overline W$ such that $C$ has non-empty interior in $Y$.  Let $\varepsilon>0$ such that $H\setminus K$, $K\setminus H$, and $C$ each contain open balls of radius $\varepsilon$.  Let $E_0$ and $E_1$ be any two $\varepsilon$-dense subcontinua of $Y$.  By hereditary unicoherence of $Y$,  $$M:=(\overline C \cup E_0\cup E_1)\cap (H\cap K)$$  is connected. Note that $M=(E_0\cap H\cap K)\cup (E_1\cap H\cap K)$, where each set in that union is non-empty and closed. Therefore $E_0\cap E_1\neq\varnothing$.\hfill$\blacksquare$

\medskip

We now prepare to prove Theorem 2. Following \cite[Definition 4.5]{and}, if  $X$ is a semicontinuum, $K\subset X$, and $\mathcal U$ is a finite collection of open subsets of $X$, then we say $K$ \textit{disrupts} $\mathcal U$ if no continuum in $X\setminus K$ intersects each member of $\mathcal U$. 

\begin{ul}If $X$ is an indecomposable semicontinuum, then no finite collection of non-empty  open subsets of $X$ is disrupted by (the union of) finitely-many proper continua $$K_0,K_1,\ldots,K_{n-1}\subsetneq X.$$\end{ul}

\begin{proof}Let $X$ be an indecomposable semicontinuum.  Let $K_0,K_1,\ldots,K_{n-1}\subsetneq X$ be continua.  Suppose for a contradiction that $K:=\bigcup \{K_i:i<n\}$ disrupts a finite collection of non-empty open sets. Let $l$ be the least positive integer with the property that some collection of non-empty open sets of size $l$ is disrupted by $K$. That is,  $$l=\min\{|\mathcal U|: \mathcal U \text{ is a collection of non-empty open subsets of }X\text{, and }K\text{ disrupts }\mathcal U\}.$$ Since $K$ is nowhere dense,  $l\geq 2$. Let $\mathcal V=\{V_0,V_1,\ldots,V_{l-1}\}$ be a collection of  non-empty open sets such that $K$ disrupts $\mathcal V$.  By minimality and finiteness of $l$,  
$$N:=\bigcup\{M\subset X\setminus K:M\text{ is a continuum and }M\cap V_j\neq\varnothing\text{ for each }j\geq 1\}$$ contains a dense subset of $V_1$. 

We claim that every constituent $M\subset N$ is contained in a semicontinuum $S\subset N$ such that $\overline S$ intersects some $K_i$. To see this, fix $p\in M$ and  $q\in  V_0$.  Since $X$ is a semicontinuum, there is a continuum $L\subset X$ such that $\{p,q\}\subset  L$. The assumption  $K$ disrupts $\mathcal V$ implies $(M\cup L)\cap K\neq\varnothing$, whence $L\cap K\neq\varnothing$. Boundary bumping  \cite[Lemma 6.1.25]{eng} in $L$ now shows that for each $i<\omega$ there is a continuum $L_i\subset L\setminus K$ such that $p\in L_i$ and $d(L_i,K)<2^{-i}$.   The semicontinuum $S:=\bigcup\{M\cup L_i:i<\omega\}$ is contained in $N$,  and $\overline {S}\cap K\neq\varnothing$ by compactness of $K$. We conclude that  $$N':=\bigcup\{\overline S:S\text{ is a maximal semicontinuum in } N\}$$ has at most $n$ connected components.   As $V_1\subset  \overline{N'}$,   this implies some component $C$ of $N'$  is dense in a non-empty open subset of $V_1$.  Then   $\overline C$ is a closed connected subset of $X\setminus V_0$ with non-empty interior. This violates indecomposability of $X$.\end{proof}

\subsection*{Proof of Theorem 2} Let $\{U_i:i<\omega\}$ be a basis for $X$ consisting of non-empty open sets. Put $\mathcal U_n=\langle U_0,\ldots,U_{n-1}\rangle$. Let $K_0\subsetneq X$ be any continuum. Assuming  mutually disjoint $K_0,\ldots,K_{n-1}$ have been defined so that $K_i\in \mathcal U_i$ for each $i<n$, by Lemma 1 there exists $K_n\in \mathcal U_n$ such that $K_n\cap K_i=\varnothing$ for each $i<n$. The sequence $(K_n)$ is as desired. \hfill$\blacksquare$

\subsection*{Proof of Theorem 3}Suppose $X$ is homeomorphic to a composant of  indecomposable continuum $I$.  Let $Y=\gamma X$ be any compactification of $X$ with associated embedding $\gamma:X\hookrightarrow \gamma X$.  Let  $\iota:X\hookrightarrow I$ be a homeomorphic embedding such that $\iota[X]$ is a composant of $I$.  Let $Z$ be the closure of the diagonal $\{\langle \iota(x),\gamma(x)\rangle:x\in X\}$ in the product $I\times Y$. More precisely,  define  $\xi:X\hookrightarrow I\times Y$ by $\xi(x) =\langle \iota(x),\gamma(x)\rangle$ and put $Z=\overline{\xi[X]}$. By the proof of \cite[Theorem 1.1]{lip4}, $Z$ is an indecomposable continuum and $\xi[X]$ is a composant of $Z$. By Lavrentiev's Theorem \cite[Theorem 4.3.21]{eng}, the homeomorphism $\pi_Y\restriction\xi[X]$  extends to a homeomorphism between $G_\delta$-sets $Z'\subset Z$ (with $\xi[X]\subset Z'$) and $Y'\subset Y$. By \cite[Theorem 9]{cook}, $Z'$ contains $\mathfrak c$ composants of $Z$.  Thus $Y'$ contains $\mathfrak c$  pairwise disjoint semicontinua which are dense in $Y$.\;\hfill$\blacksquare$

\section{Question}

  A \textit{ray} is a homomorphic image of the interval $[0,\infty)$. If $h:[0,\infty)\to X$ is a homeomorphism, then $X$ is a ray which \textit{limits onto itself} if $h([n,\infty))$ is dense in $X$ for every  $n<\omega$. This is equivalent to saying $X$ is first category in the sense of Baire.  If $Y$ is a one-dimensional non-separating plane continuum which is the closure of a ray that limits onto itself, then $Y$ is indecomposable  \cite[Theorem 8]{cur}.

\begin{uq}If $Y$ is a continuum in the plane which contains first category ray (limiting onto itself), then is $Y$ non-Suslinian? \end{uq}

%Here we ask whether Proposition 1 is true for two other special types of continua.

%\begin{uq}If $Y$ is a continuum which contains an indecomposable semicontinuum, then is $Y$ non-Suslinian? \end{uq}

%Note that by Theorem 1, the hypothesis of Question 1 is stronger than the hypothesis in Proposition 1. The answer is ``yes'' for hereditarily unicoherent continua;  in this case it is easy to show that $Y$ must be indecomposable.

%\begin{uq}If $Y$ is a plane continuum which contains a first category ray, then is $Y$ non-Suslinian? \end{uq}

%In either question if we assume that $Y$ is hereditarily unicoherent, then it is easy to show that $Y$ is indecomposable.

%Question 2 could be related to Question 1, as it may be true that every first category plane ray is indecomposable.  Finally, we ask whether Proposition 1 is true for tree-like continua.

\end{document}